\documentclass[12pt]{amsart}
\usepackage{amsmath, amssymb, amsthm, amsfonts, enumerate, verbatim}
\usepackage[all]{xy}
\usepackage{tikz}
\tikzstyle{punkt}=[circle, fill=black, minimum size=1mm,inner sep=0pt, draw]
\def\frk{\frak}               

\def\Phi{{\frk n}}
\def\Phi{{\frk N}}
\def\MB{{\mathcal B}}

\def\opn#1#2{\def#1{\operatorname{#2}}} 
\opn\chara{char}
\opn\length{\ell}
\opn\pd{pd}
\opn\rk{rk}
\opn\projdim{proj\,dim}
\opn\injdim{inj\,dim}
\opn\rank{rank}
\opn\depth{depth}
\opn\grade{grade}
\opn\height{height}
\opn\embdim{emb\,dim}
\opn\codim{codim}

\opn\Tr{Tr}
\opn\bigrank{big\,rank}
\opn\superheight{superheight}
\opn\lcm{lcm}
\opn\trdeg{tr\,deg}
\opn\reg{reg}
\opn\lreg{lreg}
\opn\ini{in}
\opn\lpd{lpd}
\opn\size{size}
\opn\bigsize{bigsize}
\opn\cosize{cosize}
\opn\bigcosize{bigcosize}
\opn\sdepth{sdepth}
\opn\sreg{sreg}
\opn\link{link}
\opn\fdepth{fdepth}
\opn\lin{lin}
\opn\ini{in}
\opn\div{div}
\opn\Div{Div}
\opn\cl{cl}
\opn\Cl{Cl}
\opn\Spec{Spec}
\opn\Supp{Supp}
\opn\supp{supp}
\opn\Sing{Sing}
\opn\Ass{Ass}
\opn\Min{Min}
\opn\Mon{Mon}
\opn\dstab{dstab}
\opn\astab{astab}
\opn\Syz{Syz}
\opn\Ann{Ann}
\opn\Rad{Rad}
\opn\Soc{Soc}
\opn\Im{Im}
\opn\Ker{Ker}
\opn\Coker{Coker}
\opn\Am{Am}
\opn\Hom{Hom}
\opn\Tor{Tor}
\opn\Ext{Ext}
\opn\End{End}
\opn\Aut{Aut}
\opn\id{id}

\opn\nat{nat}
\opn\pff{pf}
\opn\Pf{Pf}
\opn\GL{GL}
\opn\SL{SL}
\opn\mod{mod}
\opn\ord{ord}
\opn\Gin{Gin}
\opn\Hilb{Hilb}
\opn\sort{sort}
\opn\initial{init}
\opn\ende{end}
\opn\height{height}
\opn\type{type}
\opn\mdeg{mdeg}
\opn\aff{aff}
\opn\con{conv}
\opn\relint{relint}
\opn\st{st}
\opn\lk{lk}
\opn\cn{cn}
\opn\core{core}
\opn\vol{vol}
\opn\link{link}
\opn\star{star}
\opn\lex{lex}
\opn\sign{sign}
\opn\gr{gr}

\def\pot#1#2{#1[\kern-0.28ex[#2]\kern-0.28ex]}

\opn\dirlim{\underrightarrow{\lim}}
\opn\inivlim{\underleftarrow{\lim}}
\let\to=\rightarrow

\def\Implies{\ifmmode\Longrightarrow \else
	\unskip${}\Longrightarrow{}$\ignorespaces\fi}
\def\implies{\ifmmode\Rightarrow \else
	\unskip${}\Rightarrow{}$\ignorespaces\fi}
\def\iff{\ifmmode\Longleftrightarrow \else
	\unskip${}\Longleftrightarrow{}$\ignorespaces\fi}

\let\:=\colon
\newtheorem{Theorem}{Theorem}[section]

\newtheorem{Corollary}[Theorem]{Corollary}
\newtheorem{Proposition}[Theorem]{Proposition}
\newtheorem{Remark}[Theorem]{Remark}

\newtheorem{Example}[Theorem]{Example}

\newtheorem{Definition}[Theorem]{Definition}

\let\epsilon\varepsilon
\let\kappa=\varkappa
\def\pnt{{\raise0.5mm\hbox{\large\bf.}}}

\begin{document}
	\title{Combinatorial upper bounds for the smallest eigenvalue of a graph}
	
	\author {Aryan Esmailpour, Sara Saeedi Madani, Dariush Kiani}
	
\address{Aryan Esmailpour, Department of Mathematics and Computer Science, Amirkabir University of Technology (Tehran Polytechnic), Iran}
\email{aryanesmailpour@gmail.com}

\address{Sara Saeedi Madani, Department of Mathematics and Computer Science, Amirkabir University of Technology (Tehran Polytechnic), Iran, and School of Mathematics, Institute for Research in Fundamental Sciences (IPM), Tehran, Iran}
\email{sarasaeedi@aut.ac.ir}

\address{Dariush Kiani, Department of Mathematics and Computer Science, Amirkabir University of Technology (Tehran Polytechnic), Iran}
\email{dkiani@aut.ac.ir}

	\begin{abstract}
		Let $G$ be a graph, and let $\lambda(G)$ denote the smallest eigenvalue of $G$. First, we provide an upper bound for $\lambda(G)$ based on induced bipartite subgraphs of $G$. Consequently, we extract two other upper bounds, one relying on the average degrees of induced bipartite subgraphs and a more explicit one in terms of the chromatic number and the independence number of $G$. In particular, motivated by our bounds, we introduce two graph invariants that are of interest on their own. Finally, special attention goes to the investigation of the sharpness of our bounds in various classes of graphs as well as the comparison with an existing well-known upper bound.  
	\end{abstract}

	
	\maketitle
	
	\section{Introduction}\label{introduction}
	
Understanding the adjacency eigenvalues of a graph in terms of combinatorial invariants and properties of the graph has been an important problem in spectral graph theory. In particular, the smallest eigenvalue of a graph has been of special interest to many authors, see for example \cite{Alon}, \cite{Cioaba}, \cite{Nikiforov}, \cite{Nikiforov2}. There are many papers in the literature that study both lower bounds \cite{Alon}, \cite{Cioaba}, \cite{Cioaba2}, and upper bounds \cite{Brouwer2}, \cite{Nikiforov} for the smallest eigenvalue. 

The smallest eigenvalue of a graph is known to be closely related to its bipartite subgraphs. This approach has taken the attention of several authors, see \cite{Min1}, \cite{Min2}, \cite{Kwan}, \cite{Nikiforov}, and \cite{Nikiforov1}. Moreover, because of this connection, it is natural to consider the smallest eigenvalue as a measure of, roughly speaking, how bipartite the graph is. The size of a graph's maximum cut is one common way to interpret its level of bipartiteness. Therefore, many approximation algorithms for the so-called Max Cut problem which is an NP-hard problem have been provided through the aforementioned point of view, see for example~\cite{Alon}, \cite{Goemans}, \cite{NPHard} and~\cite{Trevisan}.  

Aside from the size of the maximum cut, other invariants related to the smallest eigenvalue have been defined to measure bipartiteness of a graph \cite{Min1}, \cite{Min2}, \cite{Brouwer}. In this paper, we define two additional invariants and study their connection to the smallest eigenvalue. We focus on the smallest eigenvalue of a graph and provide some combinatorial upper bounds for it. Then we extend our study by providing several cases for which our bounds are sharp as well as comparing one of our bounds with a nice bound due to Nikiforov \cite{Nikiforov}. 
	
	The organization of this paper is as follows. In Section~\ref{bound}, we prove our main result which gives an upper bound for the smallest eigenvalue of a graph~$G$, see Theorem~\ref{hndsi}. This bound is based on the induced bipartite subgraphs of $G$. Motivated by this bound, we define a new graph invariant of $G$, denoted by $\eta(G)$, see Definition~\ref{mxavdg}. Obtaining the exact value of this invariant for any graph is not easy, but Corollary~\ref{hndsicor} gives an upper bound for it. The same corollary also improves the bound given in Theorem~\ref{hndsi}. As a consequence of Theorem~\ref{hndsi}, we obtain another upper bound relying on the average degrees of induced bipartite subgraphs of~$G$. Motivated by this result, we also define another graph invariant for $G$, denoted by~$\iota(G)$, see~Definition~\ref{avdegdef}. As a first step to compute this new invariant for any graph, Corollary~\ref{avdegcol} gives an upper bound for it by improving Theorem~\ref{avdeg}. It is clear that $\iota(G)$ is bounded above by $\mathrm{mad}(G)$, the so-called maximum average degree of $G$. But, in Corollary~\ref{avdegcol}, $\iota(G)$ can not be replaced by $\mathrm{mad}(G)$. We also extract a more explicit upper bound for the smallest eigenvalue of~$G$ in Theorem~\ref{tm}. The chromatic number and the independence number of $G$ play an important role in this bound. 
	
	In Section~\ref{sharp}, we discuss several cases for which the bounds given in Section~\ref{bound} are sharp. First, we show that the class of graphs for which $\iota(G)$ is sharp is closed under the Cartesian product of graphs. Next, we focus on certain classes of graphs and evaluate the bound given in Theorem~\ref{tm} for them. For this purpose, we consider regular bipartite graphs and complete multipartite graphs with equal parts. In addition, we construct a family of graphs, based on the join product of graphs, for which the bounds given in Corollary~\ref{hndsicor} and Corollary~\ref{avdegcol} are not sharp. Indeed, in Example~\ref{badexmpl} we show that the difference between $- \eta(G)$ and $\lambda(G)$ can grow to $O(|V(G)|)$ in general. 
	
	In Section~\ref{cmp}, we compare our bound given in Theorem~\ref{tm} with a nice well-known bound given in~\cite{Nikiforov} by Nikiforov. Indeed, we show that if $\chi(G)=2$~or~$3$, or if $G$ is planar, then our bound is less than or equal to Nikiforov's, see~Theorem~\ref{comparison}. From an asymptotic point of view, in Corollary~\ref{almost all}, we deduce that the same result holds for almost all $K_t$-free graphs for $t=3,4$.               
	
	Throughout the paper, all graphs are finite and simple, (i.e. with no loops, multiple and directed edges), and by eigenvalues of a graph, we mean eigenvalues of its adjacency matrix.

	\section{Upper bounds for the smallest eigenvalue of a graph}\label{bound}
	
	Let $G$ be a graph on $n$ vertices with the adjacency matrix $A$ and (adjacency) eigenvalues $\lambda_{1}(G) \geq \lambda_{2}(G) \geq \cdots \geq \lambda_{n}(G)$. For simplicity, we denote $\lambda_{n}(G)$ by $\lambda(G)$. As usual, we denote the set of vertices and the set of edges of $G$ by $V(G)$ and $E(G)$, respectively. We denote the set of all induced bipartite subgraphs of $G$ by $\MB(G)$. In this section, we provide some upper bounds for $\lambda(G)$ in terms of certain invariants of $G$. The following theorem is our first result. 
	
	\begin{Theorem}\label{hndsi}
	Let $G$ be a graph with at least two vertices. If $H \in \MB(G)$ and $V(H) = V_1 \cup V_2$ is a bipartition of the vertices of $H$. Then
	\[
	\lambda(G) \leq \frac{-|E(H)|}{\sqrt{|V_1| |V_2|}}.
	\]
	\end{Theorem}
	\begin{proof}
	We define the vector $\mathbf{x} = (x_1, x_2, ..., x_n)^T$ such that
	    \[
	    x_i =
\left\{
	\begin{array}{ll}
		\frac{1}{\sqrt{|V_1|}}  & \mbox{if } i \in V_1 \\[3mm]
		\frac{-1}{\sqrt{|V_2|}}  & \mbox{if } i \in V_2 \\[3mm]
		0 & \mbox{otherwise } 
	\end{array}
\right.
	    \]
	    for any $i = 1,\ldots, n$. It is clear that $\mathbf{x}^T \mathbf{x} = 2$. On the other hand, it is easily seen that 
	    \[
	    \mathbf{x}^T A \mathbf{x} = 2\sum_{\{i,j\} \in E(G)} x_i x_j.
	    \]
	    Therefore, by using the Rayleigh quotient, we have
	    \begin{equation}\label{eq11}
	    \lambda(G) \leq \frac{\mathbf{x}^T A \mathbf{x}}{\mathbf{x}^T \mathbf{x}}~ = ~\frac{\mathbf{x}^T A \mathbf{x}}{2}~ = \sum_{\{i,j\} \in E(G)} x_i x_j.
	    \end{equation}
	    By the definition of $\mathbf{x}$, we have $x_i x_j \neq 0$ if and only if $i,j \in V_1 \cup V_2$. So, we get
	    \begin{equation}\label{eq12}
	    \sum_{\{i,j\} \in E(G)} x_i x_j ~=~ \sum_{\{i,j\} \in E(H)} x_i x_j ~=~ \sum_{\{i,j\} \in E(H)} \frac{-1}{\sqrt{|V_1||V_2|}}
	    =~ \frac{-|E(H)|}{\sqrt{|V_1||V_2|}}.
	    \end{equation}
	    By (\ref{eq11}) and (\ref{eq12}), we get
	    \[
	    \lambda(G) \leq \frac{-|E(H)|}{\sqrt{|V_1| |V_2|}},
	    \]
	    as desired.
	\end{proof}
	
	Motivated by Theorem \ref{hndsi}, we define the following invariant for graphs:
	
	\begin{Definition}\label{mxavdg}
	Let $G$ be a graph with at least two vertices. Then we define
	\[
	\eta(G) = \max \left\{ {\frac{|E(H)|}{\sqrt{|V_1| |V_2|}}}: \text{$H \in \MB(G)$ with the bipartition $V_1 \cup V_2$} \right\}.
	\]
	\end{Definition}
	
	As an immediate consequence of Theorem \ref{hndsi}, we get 
	
	\begin{Corollary}\label{hndsicor}
	Let $G$ be a graph with at least two vertices. Then
	\[
	\lambda(G) \leq -\eta(G).
	\]
	\end{Corollary}
	Recall that the average degree of $G$, denoted by $\overline{d}(G)$ is $\frac{2|E(G)|}{|V(G)|}$. Using the notation in Theorem \ref{hndsi}, we have
	\[
	\overline{d}(H) ~=~ \frac{2|E(H)|}{|V_1| + |V_2|} ~=~ \frac{|E(H)|}{\frac{1}{2}(|V_1| + |V_2|)} \leq \frac{|E(H)|}{\sqrt{|V_1| |V_2|}}.
	\]
	Therefore, we get
	\begin{Theorem}\label{avdeg}
	Let $G$ be a graph with at least two vertices and $H \in \MB(G)$. Then
	\[
    \lambda(G) \leq -\overline{d}(H).
	\]
	\end{Theorem}
	
	\begin{Remark}\label{rmrk}
	\em Note that the statement of Theorem \ref{avdeg} can also be obtained by another argument independent of the proof of Theorem \ref{hndsi}. Indeed, first note that it is well known for every graph $T$ that $\lambda_{1}(T) \geq \overline{d}(T)$. We use the notation of Theorem \ref{hndsi}. Since $H$ is a bipartite graph, we have $\lambda(H) = -\lambda_{1}(H)$. So we get  
	\[
	\lambda(H) \leq -\overline{d}(H).
	\]
	Since $H$ is an induced subgraph of $G$, by using the interlacing theorem, it can be seen that
	\[
	\lambda(G) \leq -\overline{d}(H).
	\]
	\end{Remark}
	
	Motivated by Theorem \ref{avdeg}, we also define the following invariant for graphs.
	
	\begin{Definition}\label{avdegdef}
	Let $G$ be a graph with at least two vertices. Then we define 
	\[
	\iota(G) = \max \{\overline{d}(H): H \in \MB(G)\}.
	\]
	\end{Definition}
	
	The following corollary immediately follows from Theorem \ref{avdeg}.
	
	\begin{Corollary}\label{avdegcol}
	Let $G$ be a graph with at least two vertices. Then
	\[
	\lambda(G) \leq -\iota(G).
	\]
	\end{Corollary}

Recall that the \emph{maximum average degree} of a graph $G$, denoted by $\mathrm{mad}(G)$ is defined as 
\[
\mathrm{mad}(G)=\max \{\overline{d}(H): H~\text{is a nonempty subgraph of}~G\}.
\]
This invariant of a graph has been considered by several authors for various purposes, see for example \cite{MAD1} and \cite{MAD2}. It is clear that 
\[
\iota(G)\leq \mathrm{mad}(G).
\]
In contrast to $\iota(G)$, the maximum average degree of $G$ does not provide a result similar to Corollary~\ref{avdegcol}. For example, if $G$ is a complete graph with $n\geq 3$ vertices, then $\lambda(G)=-1$ while $\mathrm{mad}(G)=n-1$.     
	
\medskip	
	Applying Theorem \ref{avdeg}, we give a more explicit upper bound for $\lambda(G)$ in the next theorem. Let $\alpha(G)$ be the independence number of $G$ and $\chi(G)$ be the chromatic number of $G$. We define
	\[
	\theta(G) ~=~ \min \{n/2,~ \alpha(G) \}.
	\]
	
	\begin{Theorem}\label{tm}
		Let $G$ be a graph with at least one edge. Then
		\[
	    \lambda(G) \leq \frac{-|E(G)|}{\binom{\chi(G)}{2}  ~\theta(G)}.
	    \]
	\end{Theorem}
	
	\begin{proof}
	    First note that one can see that $G$ has an induced bipartite subgraph $H$
	    such that
	    \begin{equation}\label{eq21}
	    |E(H)| \geq \frac{|E(G)|}{\binom{\chi(G)}{2}}.
	    \end{equation}
	    
	    Suppose $V(H) = V_1 \cup V_2$ is a bipartition of the vertices of $H$.
	    Since $V_1$ and $V_2$ are independent sets in $G$, we have $|V_1|,~|V_2| \leq \alpha(G)$, and hence
	    \[
	    |V_1| + |V_2| \leq 2\theta(G).
	    \]
	    So we have
	    \begin{equation}\label{eq22}
	    \overline{d}(H) ~=~ \frac{2|E(H)|}{|V_1| + |V_2|} \geq \frac{|E(H)|}{\theta(G)}.
	    \end{equation}
	    
	    Therefore, by (\ref{eq21}) and  (\ref{eq22}), together with Theorem \ref{avdeg}, we deduce that
	    \[
	    \lambda(G) \leq \frac{-|E(G)|}{\binom{\chi(G)}{2}  ~\theta(G)},
	    \]
	    as desired.
	\end{proof}
	
	\section{Evaluation on various classes of graphs}\label{sharp}
	
	In this section, we provide several cases where the given bounds in Section~\ref{bound} are sharp.
	
	\subsection{Cartesian Product of Graphs:}
	Let $G_1$ and $G_2$ be two graphs. Recall that the \emph{Cartesian product} of $G_1$ and $G_2$, denoted by $G_1\Box G_2$, is a graph whose vertex set is
	\[
	V(G_1\Box G_2)=\{(v,w):v\in V(G_1), w\in V(G_2)\},
	\]
	and two vertices $(v_1,w_1)$ and $(v_2,w_2)$ are adjacent in $G_1\Box G_2$ if either $v_1=v_2$ and $\{w_1,w_2\}\in E(G_2)$ or $w_1=w_2$ and $\{v_1,v_2\}\in E(G_1)$. Now, suppose that the bound provided in Corollary~\ref{avdegcol} is sharp for both $G_1$ and $G_2$. We show that it is also sharp for $G_1\Box G_2$. Roughly speaking, the set of graphs for which the aforementioned bound is sharp, is closed under the Cartesian product.   
	
	\begin{Proposition}\label{iotaproduct}
	Let $G_1$ and $G_2$ be two graphs with $\lambda(G_1) = -\iota(G_1)$ and $\lambda(G_2) = -\iota(G_2)$. Then
	\[
	\lambda(G_1 \Box G_2)=-\iota(G_1 \Box G_2).
	\]
	\end{Proposition}
	\begin{proof}
	Let $H_1$ and $H_2$ be the induced bipartite subgraphs of $G_1$ and $G_2$ corresponding to $\iota(G_1)$ and $\iota(G_2)$, respectively. It is well-known that 
	\[
	\lambda(G_1 \Box G_2) = \lambda(G_1) + \lambda(G_2).
	\]
	Since $\lambda(G_1) = -\iota(G_1)$ and $\lambda(G_2) = -\iota(G_2)$, we have
	\begin{equation}\label{eq30}
	\lambda(G_1 \Box G_2) = -(\iota(G_1) + \iota(G_2)).
	\end{equation}
	
	Since $H_1$ and $H_2$ are induced bipartite subgraphs of $G_1$ and $G_2$, one can see that $H_1 \Box H_2$ is also an induced bipartite subgraph of $G_1 \Box G_2$ with 	
	\begin{equation}\label{eq31}
	|E(H_1 \Box H_2)| = |E(H_1)||V(H_2)| + |E(H_2)||V(H_1)|
	\nonumber
	\end{equation}
	and
	\begin{equation}\label{eq32}
	|V(H_1 \Box H_2)| = |V(H_1)| |V(H_2)|.
	\nonumber
	\end{equation}
	 Therefore, we have
	 
	\begin{eqnarray*}
	\overline{d}(H_1 \Box H_2)
	&=& 
	\frac{2(|E(H_1)||V(H_2)| + |E(H_2)||V(H_1)|)}{|V(H_1)| |V(H_2)|}
	\nonumber\\
	&=& \frac{2|E(H_1)|}{|V(H_1)|} + \frac{2|E(H_2)|}{|V(H_2)|}
	= \iota(G_1) + \iota(G_2).
	\end{eqnarray*}
	Thus, by (\ref{eq30}) we get
	\[
	\lambda(G_1 \Box G_2) = -\overline{d}(H_1 \Box H_2).
	\]
	Since $\lambda(G_1 \Box G_2) \leq -\iota(G_1 \Box G_2)$ by Corollary~\ref{avdegcol}, and since $\iota(G_1 \Box G_2) \geq \overline{d}(H_1 \Box H_2)$, we deduce that
	\[
	\lambda(G_1 \Box G_2) ~=~ -\iota(G_1 \Box G_2),
	\]
	as desired.
	\end{proof}
	
	\subsection{Regular bipartite graphs:}\label{dbip}
	Let $d \geq 1$ and let $G$ be a $d$-regular bipartite graph with $2t$ vertices. It is well-known that $\lambda(G)=-d$. On the other hand, it is a well-known consequence of Hall's ~Marriage ~Theorem that any regular bipartite graph has a perfect matching. Therefore, it is easy to see that $\alpha(G) = t$, and hence $\theta(G)=t$. Since $\chi(H) = 2$, the given bound in Theorem~\ref{tm} turns to be 
	\[
	\frac{-|E(G)|}{\binom{\chi(G)}{2}  ~\theta(G)}=\frac{-t d}{t}=-d.
	\]
	Thus, for any $d$-regular bipartite graphs this bound is sharp.
	
	
	

	\subsection{Complete multipartite graphs with equal parts:}\label{ckpg}
	
	Let $G$ be a complete $k$-partite graph whose each part consists of $t$ vertices, where $k\geq 2$ and $t\geq 1$. Then $\lambda(G)=-t$. For the sake of convenience of the reader, we briefly argue this fact. If $k=2$, then $G$ is a complete bipartite graph and it is a well-known fact that $\lambda(G)=-t$. If $k\geq 3$, then by considering $G$ as the join of a complete $(k-1)$-partite graph with $t$ vertices in each part and the complement of a complete graph over $t$ vertices, one can use the formula given in~\cite[Theorem~2.1.8]{CRS}. Then, by using induction on $k$, it follows that $\lambda(G)=-t$.   
	

	
	 On the other hand, it is obvious that $\theta(G)=\alpha(G) = t$, $\chi(G) = k$ and $|E(G)|={k \choose 2} t^2$. Therefore, the upper bound given in Theorem~\ref{tm} turns to be 
	 \[
	 \frac{-|E(G)|}{\binom{\chi(G)}{2}  ~\theta(G)}=-t,
	 \] 	
	which shows that this bound is sharp for any multipartite graph with equal parts.
	
\subsection{A family of graphs with nonsharp bounds}
 
 We close this section with the following example which provides a family of dense graphs for which the bounds given in Corollary~\ref{hndsicor} and Corollary~\ref{avdegcol} are not sharp. Indeed, the difference between 
 $- \eta(G)$ and $\lambda(G)$ can grow to $O(|V(G)|)$ in general.


	\begin{Example}\label{badexmpl}
		{\em Given a positive integer $s$, let $H_s=2K_s*2K_s$ be the join of two copies of the disjoint union of two complete graphs~$K_s$. More precisely, $H_s$ is the graph with the vertex set $V(H_s) = V_1 \cup V_2 \cup V_3 \cup V_4$ and the edge set  
		\[
	     E(H_s)=\{ \{u, v\} \mid u \neq v, u,v \in V_i, i \in \{1, 2, 3, 4\}\} \cup \{ \{u, v\} \mid u \in V_1 \cup V_2, v \in V_3 \cup V_4\},
		\]
		where $V_1 = \{1, 2, \dots, s\}$, $V_2 = \{s + 1, s + 2, \dots, 2s\}$, $V_3 = \{2s + 1, 2s + 2, \dots, 3s\}$ and $V_4 = \{3s + 1, 3s + 2, \dots, 4s\}$.  Then, we show that 
		\[
		|-\eta(H_s) - \lambda(H_s)| \in O(|V(H_s)|).
		\]
		It is easy to see that every induced bipartite subgraph of $H_s$ has at most four vertices, and hence $\eta(H_s) = 2$. 
		On the other hand, the sets $V_i$'s for $i=1, 2, 3, 4$, provide an equitable partition for $H_s$ with the following divisor matrix 
		\[
			E =\begin{bmatrix}
			s-1 & 0 & s & s \\
			0 & s-1 & s & s \\
			s & s & s-1 & 0 \\
			s & s & 0 & s-1
			\end{bmatrix}
			= (s-1)I + s\begin{bmatrix}
			0 & 0 & 1 & 1 \\
			0 & 0 & 1 & 1 \\
			1 & 1 & 0 & 0 \\
			1 & 1 & 0 & 0
			\end{bmatrix}.
		\]
		One can see that the least eigenvalue of $E$ is equal to $-(s+1)$. Hence, we have 
		\[
		\lambda(H_s) \leq -(s+1) = -(\frac{|V(H_s)|}{4} + 1),
		\] 
		see for example~\cite[Theorem~3.9.5]{CRS}. This implies that 
		$|-\eta(H_s) - \lambda(H_s)| \in O(|V(H_s)|)$, as desired. 
		
		We would like to remark on an alternative tool for the above discussion on $\lambda(H_s)$. Indeed, similar to subsection~\ref{ckpg}, applying the formula given in \cite[Theorem~2.1.8]{CRS} yields the precise value of $\lambda(H_s)$. More precisely, it can be seen that in fact we have $\lambda(H_s) = -(\frac{|V(H_s)|}{4} + 1)$.}
	\end{Example}
	
		\section{A comparison with Nikiforov's bound}\label{cmp}
	
	In this section, we compare the bound provided in Theorem~\ref{tm} with the one given in the following theorem by Nikiforov in~\cite{Nikiforov} in several cases. 
	
	\begin{Theorem}\label{Nik}
		\cite[Theorem~1]{Nikiforov}
		Let $G$ be a $K_{r+1}$-free graph on $n$ vertices and $m$ edges where $r\geq 2$. Then
		\[
		\lambda(G) < - \frac{2^{r+1} m^r}{r n^{2r-1}}.
		\]
	\end{Theorem}
	
	The next theorem provides certain families of graphs for which our bound in Theorem~\ref{tm} is sharper than the one given in Theorem~\ref{Nik}. 
	
	\begin{Theorem}\label{comparison}
		Let $G$ be a graph with $n$ vertices and $m\geq 1$ edges. Suppose that one of the following conditions holds: 
		\begin{enumerate}
			\item[\em(a)] $\chi(G)=2$ or $3$;
			\item[\em(b)] $G$ is a planar graph.
		\end{enumerate}
		Then 
		\[
		-\frac{m}{\binom{\chi(G)}{2}  ~\theta(G)} \leq -\frac{2^{r+1} m^r}{r n^{2r-1}}. 
		\]
	\end{Theorem}
	
	\begin{proof}
		Suppose that $G$ is $r$-colorable. Then it is clear that $G$ is $K_{r+1}$-free and we have 
		\[
		\frac{m}{\binom{\chi(G)}{2}  ~\theta(G)} \geq \frac{2m}{r(r - 1)\theta(G)},
		\]	
		since $2\leq \chi(G)\leq r$. Thus, if each of (a) or (b) holds, then we need to show that 
		\[
		\frac{2m}{r(r - 1)\theta(G)} \geq \frac{2^{r+1} m^r}{r n^{2r-1}}
		\]
		or equivalently, 
		\begin{equation} \label{eq41}
		n^{2r-1} \geq (r-1)~\theta(G)2^{r} m^{r-1}.
		\end{equation}
		
		(a) By the well-known Tur\'an's theorem, we have $m\leq (1-\frac{1}{r})\frac{n^2}{2}$, see for example~\cite{Turan}. Therefore, by (\ref{eq41}), it suffices to show that 
		\[
		n^{2r-1} \geq (r-1)~\theta(G)2^{r} \big(\frac{(r-1)n^2}{2r}\big)^{r-1}
		\]
		or equivalently, 
		\[
		n \geq~ 2 \theta(G) ~\frac{(r-1)^r}{r^{r-1}}
		\]
		for $r=2,3$. By the definition of $\theta(G)$, it is easy to see that this is the case.  
		
		(b) Suppose that $G$ is planar. Then it is $4$-colorable, by the well-known $4$-color theorem. We have $n\geq 2$. If $n=4$ or $5$, then one can easily check that (\ref{eq41}) holds. Now, suppose that $n\neq 4,5$. Since $G$ is planar, it is also known that $m \leq 3n-6$. Therefore, to prove (\ref{eq41}), it is enough to verify the following inequality     
		\[
		n^6 \geq 24(3n-6)^3,
		\]
		since $\theta(G) \leq n/2$. One can easily see that the latter inequality holds for all desired $n\neq 4,5$. 
	\end{proof}
	
	We would like to end this section with an asymptotic comparison. It is clear that any $r$-colorable graph is $K_{r+1}$-free, but it is well-known that if $n$ is large enough, then the number of $K_{r+1}$-free graphs on $\{1,\ldots,n\}$ is strictly bigger than the number of  $r$-colorable graphs on $\{1,\ldots,n\}$. Kolaitis et al. in \cite{allmost} studied $K_{r+1}$-free graphs from an asymptotic point of view. More precisely,
	
	\begin{Theorem}\label{asymptotic}
		\cite[Theorem~1]{allmost}	
		Let $S_{n}(r)$ be the number of $K_{r+1}$-free graphs on $\{1,\ldots,n\}$, and let $L_n(r)$ be the number of labeled $r$-colorable graphs on $\{1,\ldots,n\}$. Then     
		\[
		\lim_{n \to \infty} \big(\frac{S_{n}(r)}{L_{n}(r)}\big)=1.
		\]
	\end{Theorem}
	
	Combining Theorem~\ref{comparison}, part~(a) and Theorem~\ref{asymptotic}, we immediately get the following corollary.
	
	\begin{Corollary}\label{almost all}
		Let $t=3,4$. Then, for almost all $K_t$-free graphs, the inequality given in Theorem~\ref{comparison} holds. 
	\end{Corollary}

	\section*{Acknowledgment}
	The second author was in part supported by a grant from IPM (No. 1403130020).

\end{document}